\documentclass[11pt]{article}  
 \pdfoutput=1
\usepackage{color}              
\usepackage{epsfig}
\usepackage{graphicx}
\usepackage{amssymb}
\usepackage{amsmath}
\usepackage{epsfig}
\usepackage{bm}

\newtheorem{lemma}{Lemma}

\newtheorem{Theorem}[lemma]{Theorem}

\newtheorem{Corollary}[lemma]{Corollary}
\newtheorem{definition}[lemma]{Definition}

\newtheorem{claim}[lemma]{Claim}

\newcommand{\eps}{\varepsilon}

\newcommand{\Rbold}{{\mathbb{R}}}

\newcommand{\E}{\mathbb{E}}

\newcommand{\prob}{\mathbb{P}}

\renewcommand{\phi}{\varphi}

\newcommand{\al}{\alpha}
\newcommand{\lam}{\lambda}

\renewcommand{\r}{r_{\text{max}}}
\newcommand{\rmin}{r_{\text{min}}}

\newcommand{\NGe}{N_G(X,e)}
\newcommand{\NGee}{N_G(X,e,e')}
\newcommand{\rGe}{r_G(X,e)}
\newcommand{\rei}{(\r)^{|E_i|-1}}

\renewcommand{\P}{{\bf P}}
\newcommand{\F}{\mathcal{F}}

\def\ind{{\rm 1\hspace{-0.90ex}1}}

\begin{document}
\author{Shankar Bhamidi\thanks{Mathematics Department, University of British Columbia. Research supported by N.S.F. Grant DMS 0704159 and by PIMS and NSERC Canada}\qquad
Guy Bresler\thanks{Department of Electrical Engineering and Computer Sciences, University of California, Berkeley. Research supported by a Vodafone
Fellowship and an NSF Graduate Fellowship}\qquad Allan Sly\thanks{Department of Statistics, University of California, Berkeley. Research supported by N.S.F. Grants DMS 0528488 and DMS 0548249 and by DOD ONR Grant
N0014-07-1-05-06}
}
       
       \title{Mixing time of exponential random graphs}
\date{}
\maketitle
\footnotetext{An extended abstract of this work appeared in FOCS 2008 \cite{BBS08}.}

\begin{abstract}
A variety of random graph models have been developed in recent years to study a range of problems on networks,  driven by the wide availability of data from many social, telecommunication, biochemical and other networks.  A key model, extensively used in the sociology literature, is the exponential random graph model.  This model seeks to incorporate in random graphs the notion of reciprocity, that is, the larger than expected number of triangles and other small subgraphs.  Sampling from these distributions is crucial for parameter estimation  hypothesis testing, and more generally for understanding basic features of the network model itself.  In practice sampling is typically carried out using Markov chain Monte Carlo, in particular either the Glauber dynamics or the Metropolis-Hasting procedure.

In this paper we characterize the high and low temperature regimes of the exponential random graph model. We establish that in the high temperature regime the mixing time of the Glauber dynamics is $\Theta(n^2 \log n)$, where $n$ is the number of vertices in the graph;  in contrast, we show that in the low temperature regime the mixing is exponentially slow for any local Markov chain.  Our results, moreover, give a rigorous basis for criticisms made of such models.   In the high temperature regime, where sampling with MCMC is possible, we show that any finite collection of edges are asymptotically independent; thus, the model does not possess the desired reciprocity property, and is not appreciably different from the Erd\H{o}s-R\'enyi random graph.

\end{abstract}



{\bf MSC2000 subject classification.}
60C05, 05C80, 90B15.


\section{Introduction}

In the recent past there has been explosion in the study of real-world networks including rail and road networks, biochemical networks, data communication networks such as the Internet, and social networks. This has resulted in a concerted interdisciplinary effort to develop new mathematical network models to explain characteristics of observed real world networks, such as power law degree behavior, small world properties, and a high degree of clustering (see for example \cite{newman-survey,barabasi-survey,durrett-book} and the citations therein). 

Clustering (or reciprocity) refers to the prevalence of triangles in a graph. This phenomenon is most easily motivated in social networks, where nodes represent people and edges represent relationship. The basic idea is that if two individuals share a common friend, then they are more likely than otherwise to themselves be friends. However, most of the popular modern network models, such as the preferential attachment and the configuration models, are essentially tree-like and thus do not model the reciprocity observed in real social networks. 

One network model that attempts to incorporate reciprocity is the exponential random graph model. This model is especially popular in the sociology community. The model follows the statistical mechanics approach of defining a Hamiltonian to weight the probability measure on the space of graphs,  assigning higher mass to graphs with ``desirable''  properties. While deferring the general definition of the model to Section~\ref{sec:def}, let us give a brief example. Fix parametric constants $h,\beta >0$ and for every graph $X$ on $n$ labeled vertices with $E(X)$ edges and $T(X)$ triangles,  define the Hamiltonian of the graph as 
\[H(X) = h E(X) + \beta T(X)\,.\]
A probability measure on the space of graphs may then be defined as 
\begin{equation}\label{e:gibbsDistn_Intro}p_n(X)  = \frac{e^{H(X)}}{Z}\,,\end{equation}
where $Z$ is the normalizing constant often called the partition function. More generally, one can consider Hamiltonians in graphs which include counts $T_i(X)$ of different small subgraphs $G_i$,
\[H(X) =  \sum_i \beta_i T_i(X)\,.\]

Social scientists use these models in several ways. The class of distributions \eqref{e:gibbsDistn_Intro} is an exponential family, which allows for statistical inference of the parameters using the subgraph counts (which are sufficient statistics for the parameters involved). Sociologists carry out tests of significance, hoping to understand how prescription of local quantities such as the typical number of small subgraphs in the network affects more global macroscopic properties.   
Parameter estimation can be carried out either by maximum likelihood or, as is more commonly done, by simply equating the subgraph counts.  Both procedures in general require sampling, in the case of maximum likelihood to estimate the normalizing constants.  Thus, efficient sampling techniques are key to statistical inference on such models.  At a more fundamental level, sociologists are interested in the the question of how localized phenomena involving a small number people  determine the large scale structure of the networks  \cite{snijders2004}.  Sampling exponential random graphs and observing their large scale properties is one way this can be realized.  Sampling is almost always carried out using local MCMC algorithms, in particular the Glauber dynamics or Metropolis-Hasting.  These are reversible ergodic Markov chains, which eventually converge to the stationary distribution $p_n(X)$.  However, our results show that the time to convergence can vary enormously depending on the choice of parameters.

{\bf Our results:}
It is surprising that in spite of the practical importance of sampling from exponential random graph distributions, there has been no mathematically rigorous study of the mixing time of any of the various Markov chain algorithms in this context. The goal of this paper is to fill this gap. We focus attention to the  Glauber dynamics, one of the most popular Markov chains.  We provide the first rigorous analysis of the mixing time of the Glauber dynamics for the above stationary distribution and do so in a very general setup. In the process we give a rigorous definition of the ``high temperature" phase, where the Gibbs distribution is unimodal and the Glauber dynamics converges quickly to the stationary distribution,  and the ``low temperature'' phase, where the Gibbs distribution is multimodal and the Glauber dynamics takes an exponentially long time to converge to the stationary distribution.  While a complete understanding of the Gibbs distribution in the low temperature phase remains out of reach (see, however, the important work of Sourav Chatterjee in the case of triangles \cite{sourav-rnd}), we can nevertheless show that the distribution has poor conductance, thereby establishing exponentially slow mixing for any local Markov chain with the specified stationary distribution. 

Our results, moreover, give a rigorous basis for criticisms made of such models.   In the high temperature regime, where sampling with MCMC is possible, we show that any finite collection of edges are asymptotically independent. Also, we show that with exponentially high probability a sampled graph is \emph{weakly pseudorandom}, meaning that it satisfies a number of equivalent properties (such as high edge expansion) shared by Erd\H{o}s-R\'enyi random graphs. Thus, the model does not possess the desired reciprocity property, and is not appreciably different from the Erd\H{o}s-R\'enyi random graph. 

{\bf Relevant literature:}
There is a large body of literature, especially in the social networking community, on exponential random graph models. We shall briefly mention just some of the relevant literature and how it relates to our results (see \cite{newman-survey,snijders2004,wasserman-survey} and the references therein for more background). The pioneering article in this  area by Frank and Strauss \cite{frank-strauss} introduced the concept of Markov graphs. Markov graphs are a special case of exponential random graphs with only  situation where the subgraphs are stars or triangles. Extending the methodology of \cite{frank-strauss}, Wasserman and Pattison \cite{wasserman-pattison} introduced general subgraph counts. However, from the outset a number of researchers noted problems at the empirical level for their Markov chain algorithms, depending on parameter values. See \cite{snijders2004} for a relevant discussion of  empirical findings as well as several new specifications of the model to circumvent such issues. 


On the theoretical side, Sourav Chatterjee \cite{sourav-rnd}, in his recent work characterizing the large deviation properties of Erd\H os-R\'enyi random graphs, developed  mathematical techniques that can be used to study the distribution these random graphs.  At the statistical physics (non-rigorous) level Mark Newman and his co-authors have studied the case where the subgraphs are triangles and 2-stars. In this setting, using mean-field approximations, they predicted a phase transition between a high-symmety phase, with graphs exhibiting only a mild amount of reciprocity, and a degenerate symmetry-broken phase with either high or low edge density (see \cite{newman2star} and \cite{newmancluster}).

\subsection{Definitions and Notation}
\label{sec:def}
This section contains a precise mathematical definition of the model and the Markov chain methodology used in this paper. 
We work  on the space $\mathcal{G}_n$ of all  graphs on $n$ vertices with vertex set $[n]:= \{1,2,\ldots, n\}$. We shall use $X = (x_e)$ to denote a graph from $\mathcal{G}_n$ where for every edge $e = (i,j)$,  $x_e$ is $1$ if the edge between vertex $i$ and $j$ is present and $0$ otherwise. For simplicity, we shall often write $X(e)$ for $x_e$. The exponential random graph model is defined in terms of the number of subgraphs $G$ (e.g., triangles or edges) contained in $X$ . It will be convenient to define these subgraph counts as follows.  Fix a graph $G$ on the vertex set $1,2,\ldots m$. Let $[n]^m$ denote the set of all $m$ tuples of distinct elements:
\[[n]^m:= \{ (v_1,\ldots, v_m): v_i\in [n], v_1\neq v_2\cdots \neq \ldots, v_m\}\,.\] 
We shall denote such an $m$ tuple of distinct vertices by ${\bf v}_m$. In a graph $X$, for any $m$ distinct vertices ${\bf v}_m$,  let $H_X({\bf v}_m) $ denote the subgraph of $X$ induced by ${\bf v}_m$. Say that $H_X({\bf v}_m)$ contains $G$, denoted by $H_X({\bf v}_m) \cong G$, if whenever the edge $(i,j)$ is present in $G$, then the edge $(v_i,v_j)$ is present in $H_X({\bf v}_m)$ for all $\{1\leq i\neq j \leq m\}$. 
For a configuration $X\in \mathcal{G}_n$ and a fixed graph $G$ define the count
\begin{equation}
N_{G}(X) = \sum_{{\bf v}_m \in [n]^m} \ind\{H_X({\bf v}_m) \cong G \}.
\end{equation}
This definition is equivalent to the usual exponential random graph model up to adjustments in the constants $\beta$ by multiplicative factors.  It counts subgraphs multiple times; for instance a triangle will be counted 6 times and in general a graph $G$ with $k$ automorphisms will be counted $k$ times.  By dividing the parameters $\beta_i$ by this multiplicative factor we reduce to the usual definition.

In our proof we shall also need more advanced versions of the above counts which we define now. Fix an edge $e = (a,b) \in X$. The subgraph count of $G$ in $X \cup \{e \}$  containing edge $e$ is defined as:
\[N_G(X,e) = \sum_{{\bf v}_m \in [n]^m, {\bf v}_m \ni a,b} \ind\{H_{X \cup \{e\}}({\bf v}_m) \cong G\}\,.\]
Similarly, for two edges $e =(a,b)$ and $e^\prime = (c,d)$ define the subgraph counts of $G$ in $X \cup \{e , e^\prime\} $ and containing edges $e, e^\prime$ by
\[N_G(X, e, e^\prime) =  \sum_{{\bf v}_m \in [n]^m, {\bf v}_m \ni a,b, c, d} \ind\{H_{X \cup \{e, e^\prime\} }({\bf v}_m) \cong G\}\,. \]  

{\bf Gibbs measure: } We now define the probability measure on the space $\mathcal{G}_n$. 
Fix $k\geq 1$ and fix graphs $G_1, G_2 \ldots, G_s$ with $G_i$ a graph on $|V_i|$ labelled vertices, with $|V_i|\leq L$ and with edge set $E_i$.  For simplicity we shall think of $G_i$ as a graph on the vertex set $1,2,\ldots |V_i|$.  By convention we shall always let  $G_1$ denote the edge graph consisting of the graph with vertex set  $1,2$ and edge set $(1,2)$. In this notation, for any configuration $X\in \mathcal{G}_n$ the quantity $N_{G_1}(X) $ will be twice the number of edges in $X$.  With this convention, fix constants $\beta_1, \beta_2, \ldots \beta_s$ with $\beta_i>0$ for $i\geq 2$ and $\beta_1\in \Rbold$. The exponential random graph probability measure is defined as follows. 

\begin{definition}
For  $G_1,\ldots G_s$ and constants ${\mathbf{\beta}} = (\beta_1, \ldots, \beta_s)$ as above, the Gibbs measure on the space $\mathcal{G}_n$ is defined as the probability measure
\begin{equation}
\label{eqn:gibbs}
p_n(X)  = \frac{1}{Z_n(\mathbf{\beta})} \exp\left(\sum_1^s \beta_i \frac{N_{G_i}(X)}{n^{|V_i| -2}}\right) \hspace{6mm} X\in \mathcal{G}_n\,.
\end{equation}
\end{definition}
Here $Z_n(\mathbf{\beta})$ is the normalizing factor and is often called the partition function. For simplicity we have suppressed the dependence of the measure on the vector $\mathbf{\beta}$. Also, note the normalization of the subgraph count of $G_i$ by the factor $n^{|V_i|-2}$, so that the contribution of each factor scales properly and is of order $n^2$ in the large $n$ limit. Setting $\beta_i\geq 0$ for $i\geq 2$ makes the Gibbs measure a monotone (also ferromagnetic) system which will be important for our proof.  The term $\beta_1$ does not affect the interaction between edges and plays the role of an external field in this model; adjusting $\beta_1$ makes it more or less likely for edges to be included. 

The term in the exponent is often called the {\it Hamiltonian} and we shall denote it by:
\[H(X) = \sum_1^s \beta_i \frac{N_{G_i}(X)}{n^{|V_i| -2}} \,.\]
Note that $H(X): \{0,1\}^{{n \choose 2}} \to \Rbold^+$ is a function of ${n \choose 2}$ Boolean variables $X(e)$ and has an elementary Fourier decomposition in terms of the basis functions $\prod_{e\in S} X(e)$, where $S$ runs over all possible subsets of edges. Thus, with respect to any fixed edge $e$, we can decompose the above Hamiltonian as
\[H(X) = A_e(X) + B_e(X)\,,\]
where $A_e$ consists of all terms dependent on edge $e$ and $B_e(X)$ denotes all terms independent of edge $e$. Let $X_{e+}$ denote the configuration of edges which coincides with $X$ for edges $e\neq f$ and has $X_{e+}(e) = 1$. The partial derivative with respect to the edge $e$ of the Hamiltonian $H$, evaluated at a configuration $X$, is defined by the formula
\[\partial_e H(X) = A_e(X_{e+})\,. \]
The higher derivatives $\partial_e \partial_{e^\prime}$ for $e\neq e^\prime$ are defined similarly by iterating the above definition. 

{\bf Glauber dynamics and local chains: } The Glauber dynamics is an ergodic reversible Markov chain with stationary distribution $p_n(\cdot)$, where at each stage exactly one edge is updated.  It is defined as follows:
\begin{definition}\label{Glauber}
Given the Gibbs measure above,  the corresponding Glauber dynamics is a discrete time ergodic Markov chain on $\mathcal{G}_n$. Given the current state $X$, the next state $X^\prime$ is obtained by choosing an edge $e$ uniformly at random and letting $X^\prime= X_{e+}$ with probability proportional $p_n(X_{e+})$ and  $X^\prime(e) = X_{e-}$ with probability proportional to $p_n(X_{e-})$. Here $X_{e+}$ is the graph which coincides with $X$ for all edges other than $e$ and  $X_{e+}(e) = 1$.  Similarly  $X_{e-}$ is the graph which coincides with $X$ for all edges other than $e$ and  $X_{e-}(e) = 0$. 
\end{definition}

There are various other chains that can also be used to simulate the above Gibbs measure. Call a chain on $\mathcal{G}_n$ {\it local} if at most $o(n)$ edges are updated in each step.   The transition rates for the Glauber dynamics satisfy the following relation:
\begin{lemma}\label{l:transitionProbabilities}
Given that we chose edge $e$ to update,  the probability of the transition $X \hookrightarrow X_{e+}$  is 
$ \frac{\exp(\partial_e H(X))}{1+\exp(\partial_e H(X))} $
and the probability of the transition $X \hookrightarrow X_{e-}$ is
$\frac{1}{1+\exp(\partial_e H(X))} $
\end{lemma}
{\bf Mixing time:} We will be interested in the time it takes for the Glauber dynamics to get close
to the stationary distribution given by the Gibbs measure (\ref{eqn:gibbs}).
The {\em mixing time } $\tau_{mix}$ of a Markov chain is defined as the
number of steps needed in order to guarantee that the chain,
starting from an arbitrary state, is within total variation distance
$e^{-1}$ from the stationary distribution.

We mention the following fundamental result which draws a connection between  total variation distance and coupling. It allows us to conclude that if we can couple two versions of the Markov chains started from different states quickly, the chain mixes quickly. The following lemma is well known, see e.g. \cite{aldous-fill}. 

\begin{lemma}[Mixing time Lemma]
\label{lemma:mixing}
For a Markov chain $X$,  suppose there exist two coupled copies, $Y$ and $Z$, such that each is marginally distributed as $X$ and
\[
\max_{y,z} [\prob(Y_{t_0} \neq Z_{t_0}|Y_0=y,Z_0=z) \leq (2e)^{-1}.
\]
Then the the mixing time of $X$ satisfies $\tau_{mix} \leq t_0$.
\end{lemma} 

Because the exponential random graph model is a monotone system, we can couple the Glauber dynamics so that if $X (0)\leq Y(0)$, then for all $t$, $X (t)\leq Y(t)$. This inequality is a partial ordering meaning that the edge set of $X$ is a subset of the edge set of $Y$.  This is known as the monotone coupling and, by monotonicty, Lemma \ref{lemma:mixing} reduces to bounding the time until chains starting from the empty and complete graphs couple.


With the above definitions of the Gibbs measure, the following functions determine the properties of the mixing time. Define for fixed $\mathbf{\beta} \in \Rbold\times(\Rbold_+)^{s-1}$ the functions
\[
\Psi_{\mathbf{\beta}}(p) = \sum_{i=1}^s 2  \beta_i |E_i| p^{|E_i| - 1}
\]

\[
\varphi_{\mathbf{\beta}}(p) = \frac{\exp(  \Psi(p)  ) } {1 + \exp(  \Psi(p)  ) }\,.
\]

Note that $\Psi_{{\mathbf{\beta}}}$ is a smooth, strictly increasing function on the unit interval. Since $\phi_\beta(0)> 0$ and $\phi_\beta(1)<1$ the equation $\phi_\beta(p)=p$ has at least one solution, denoted by $p^*$.  If this solution is unique and not an inflection point, then $0<\phi_\beta'(p^*)<1$.  The function $\phi(p)$ has the following loose motivation: if $X$ is a graph chosen according to the Erd\H{o}s-R\'enyi distribution $G(n,p)$, then with high probability all edge update probabilities $\frac{\exp(\partial_e H(X))}{1+\exp(\partial_e H(X))} $ are approximately $\phi(p)$.

{\bf Phase identification:  } We now describe the high and low temperature phases of this model. Recall that our parameter space is $\mathcal{B} = \Rbold \times (\Rbold_+)^{s-1}$. We call $p\in [0,1]$ a fixed point if  $\varphi_{\mathbf{\beta}}(p) = p$.  
 
 {\it  High temperature phase: } We say that a $\beta \in \mathcal{B}$ belongs to the high temperature phase if $\phi_\beta(p)=p$ has a unique fixed point $p^*$ which satisfies
 \begin{equation}
 \varphi_{\beta}^\prime(p^*) <1.
 \end{equation}

{\it Low temperature phase: } We say that a $\beta \in \mathcal{B}$ belongs to the low temperature phase if $\phi_\beta(p)=p$ has at least two fixed points $p^*$ which satisfy $ \varphi_{\beta}^\prime(p^*) <1$.

Values of $\beta$ not in either phase are said to be in the critical points.  They occur when one of the fixed points is an inflection point of $\phi_\beta$.  These critical points form an $s-1$ dimensional manifold which is in the intersection of the closure of the high and low temperature phases.  For simplicity, in the proof we shall suppress the dependence of the functions on $\mathbf{\beta}$ and write $\varphi$ for $\varphi_{\mathbf{\beta}}$ and $\Psi$ for $\Psi_{\mathbf{\beta}}$.

\subsection{Results}
The first two results show that the high and low temperature phases determine the mixing time for local Markov chains. 
\begin{Theorem}[High temperature]
\label{theo:high}
 If $\varphi(p)$ is in the high temperature regime then the mixing time of the Glauber dynamics is $\Theta(n^2\log{n})$.
\end{Theorem}


\begin{Theorem}[Low temperature]
\label{theo:low}
 If $\varphi(p)$ is in the low temperature regime then the mixing time of the Glauber dynamics is $e^{\Omega(n)}$. Furthermore, this holds not only for the Glauber dynamics but for any local dynamics on $\mathcal{G}_n$.  
\end{Theorem}

The next theorem shows that the exponential random graph model is not appreciably different from Erd\H{o}s-R\'enyi random graph model in the high temperature regime where sampling is possible. 

\begin{Theorem}[Asymptotic independence of edges] \label{t:asym}
  Let $X$ be drawn from the exponential random graph distribution in the high temperature phase. 
  Let $e_1,\dots,e_k$ be an arbitrary collection of edges with associated indicator random variables $x_{e_i}=\ind(e_i\in X)$. Then for any $\eps>0$, there is an $n$ such that for all $(a_1,\dots,a_k)\in \{0,1\}^k$ the random variables $x_{e_1},\dots,x_{e_k}$ satisfy 
  $$\left|\P(x_1=a_1,\dots,x_k=a_k)-(p^*)^{\sum a_i}(1-p^*)^{k-\sum a_i}\right|\leq \frac{\eps}{n^{|V|}}\,.$$
  Thus, the random variables  $x_{e_1},\dots,x_{e_k}$ are asymptotically independent.
\end{Theorem}

A consequence is that a graph sampled from the exponential random graph distribution is with high probability weakly pseudo-random (see \cite{KrSu} or \cite{CGW89}). This means that it satisfies a number of equivalent properties, including large spectral gap and correct number of subgraph counts, that make it very similar to an Erd\H{o}s-R\'enyi random graph. 

\begin{Corollary}[Weak pseudo-randomness]
With probability $1-o(1)$ an exponential random graph is weakly pseudo-random. 
\end{Corollary}

\subsection{Idea of the proof}
We give a summary of the main ideas of the proof:
\begin{itemize}
\item Consider first the high temperature phase. A natural approach to bounding the coupling time, and hence the mixing time by Lemma \ref{lemma:mixing}, is to use the technique of {\it path coupling} \cite{bubley-dyer97}. In path coupling, instead of trying to couple from every pair of states, we try to show that for any pair of states $x$ and $y$ that differ in a single edge there exists a coupling of two copies of the chain started at $x$ and $y$ such that
\begin{equation}
\label{eqn:dob}
\E(d_H(X(1), Y(1))|X(0) = x, Y(0) =y) \leq (1-\beta) 
\end{equation}
for some $\beta = \beta(n)$, where $d_H$ is the Hamming distance. However, this approach fails for some $\phi_\beta$ in the high temperature regime when
$\sup_{0\leq p \leq 1} \phi'(p) > 1$.

\item  
It turns out that the configurations in the high temperature regime where path coupling fails are very rare under the Gibbs measure.  We therefore define a set (a neighborhood of the unique fixed point $\varphi_{\mathbf{\beta}}^\prime(\cdot)< 1$), in which path coupling does give a contraction. 
More precisely, for a configuration $X$, define 
\begin{equation}
\label{e:r_G}
\rGe= \left(\frac{\NGe}{2|E| n^{|V|-2}}\right)^{\frac{1}{|E|-1}}\,.
\end{equation}
This is (asymptotically) the maximum likelihood choice for the parameter $p$ of the Erd\"os-Renyi random graph on $n$ vertices, $G(n,p)$, having observed $\NGe$ subgraphs $G$ containing the edge $e$.  Let $\{G_\lam\}$ denote the class of all graphs with at most $L$ vertices, where $L$ is some integer greater than or equal to $\max_i |V_i|$. What we prove is that for $\eps$ small enough, if the two configurations $x$ and $y$ belong to the set
\[\mathbf{G} := \{X: \max_{G\in G_\lam\atop e\in E} |\rGe - p^*| <      \eps        \}\]
then equation \ref{eqn:dob} holds for $\beta(n) = -\delta/n^2$ for some $\delta> 0$.  Thus, starting from any state $x$, if we can show that in a small number of steps ($O(n^2)$ is enough) we reach $\mathbf{G}$, then a variant of path coupling proves rapid mixing. This preliminary stage where we run the Markov chain for some steps so that it reaches a ``good configuration'' is termed the {\it burn in} phase. This approach has been used before, particularly in proving mixing times for random colourings, for example in \cite{dyer-frieze}.

  \item 
To show that we enter the good set $\mathbf{G}$ quickly,   we control all the $\rGe$,  for all subgraphs $G\in G_\lambda$ simultaneously, and via a coupling with biased random walks show that with exponentially high probability for large $n$, within $O(n^2)$ steps we reach the set $\mathbf{G}$. We crucially make use of the monotonicity of the system here by writing the drifts in terms of the $\rGe$ and bounding them by their maximum.  This completes the proof for the rapid mixing in the high temperature phase. This also shows how in the high temperature phase, most of the Gibbs measure of the exponential random graph model is concentrated on configurations which are essentially indistinguishable from the Erd\"os-Renyi $G(n, p^*)$ random graph model.  

\item In the low temperature phase we use a conductance argument to show slow mixing for any Markov chain that updates $o(n)$ edges per time step. The argument makes use of the same random walk argument used in the burn in stage to bound the measure of certain sets of configurations under the Gibbs measure. Specifically, we show that for every fixed point $p^*$ of the  equation $\phi(p)=p$ with $\phi^\prime(p)< 1$, the Glauber dynamics allows an exponentially small flow of probability to leave the set of configurations that are nearly indistinguishable from an Erd\"os-Reny\'i random graph with parameter $p^*$. Because the stationary distribution of the Glauber dynamics is the Gibbs measure, this allows us to bound the relative measure of the sets under consideration, thereby showing that if we have two or more fixed points $p^*$, then it takes an exponentially long time for configurations to leave the set of configurations indistinguishable from an Erd\"os-Reny\'i random graph with parameter $p^*$. Thus mixing takes an exponentially long time. 

\end{itemize}


\section{Proof of the main results}
\subsection{Subgraph counts}
Before starting the proof we need a couple of simple lemmas on the subgraph counts. For a graph $X\in \mathcal{G}_n$ recall the subgraph counts $N_G(X)$ of a predefined graph $G$ on $m$ nodes as well as the counts in $X$ of the subgraphs containing edges, namely $N_G(X, e)$ and $N_G(X, e, e^\prime)$ as defined in Section \ref{sec:def}. 

%

The following lemma records the quantities $N_G(X)$, $N_G(X,e)$, and $N_G(X,e,e')$ for the complete graph $X=K_n$. 


\begin{lemma}
\label{lemma:sbg-count}
Consider the complete graph on $n$ vertices, $K_n$, and let $N_G(K_n)$, $N_G(K_n,e)$, and $N_G(K_n,e,e')$ be defined as above.  Then 
\\(a) \[N_G(K_n) = {n\choose |V|}|V|! \sim n^{|V|}   \]
(b)  
\[N_G(K_n, e) = 2|E| {n-2 \choose |V|-2}(|V|-2)!\sim 2 |E| \cdot  n^{|V| - 2}\]
(c) For a fixed edge $e$ we have
\[\sum_{e^\prime\neq e} N_G(K_n, e,e^\prime) = (|E|-1)N_G(K_n,e)\sim 2|E|(|E|-1)n^{|V| -2} \]


\end{lemma}


\begin{lemma}\label{l:countingSubgraphsSum}
   For an edge $\al$ in the graph $G$, denote by $G_\al$ the graph obtained from $G$ by removing the edge $\al$. Then  \begin{equation}\label{e:countingSubgraphsSum}
     \sum_{e'\neq e} \NGee=\sum_{\al\in E(G)\atop \al\neq e} N_{G_\al}(X,e)\,.
   \end{equation}
\end{lemma}
\begin{proof}
  The sum on the left-hand side of \ref{e:countingSubgraphsSum} counts the total number of isomorphic embeddings of $G$ that contain the edge $e$ in the configuration $X\cup \{e'\}$, for some $e'$ with the edge $e'$ marked. Now, each isomorphism with marked edge $e'$ is counted on the right-hand side of \ref{e:countingSubgraphsSum} for the choice $\al$ equal to the marked edge in the graph $G$, with the same isomorphism restricted to $G_\al$. Conversely, for each $\al\in E(G)$ and each subgraph embedding, the same embedding is counted on the left-hand side with the edge $e'$ situated at the location $\al$. 
\end{proof}


\subsection{Burn-in period}\label{s:burnIn}
In this section we show that after a suitably short ``burn in" period, the Markov chain is in the good set $\mathbf{G}$. Let $\r(X)= \max_{e,\lam} r_{G_\lam}(X,e)$. The following lemma bounds the expected drift of $\r(X)$. 

\begin{lemma} \label{l:driftCalculation}
 The expected change in $N_G(X,e)$ after one step of the Glauber dynamics, starting from the configuration $X$, can be bounded as
    $$\E\left[ \frac{N_G(X(1),e) - N_G(X(0),e)}{n^{|V|-2}}\right]\leq (1+o(1))\frac{2}{{n\choose 2}} |E| (|E|-1) [-r(G,e)^{|E|-1}+\phi(\r)(\r)^{|E|-2}]\,.$$
\end{lemma}
\begin{proof}
For ease of notation, we suppress the dependence of $\r$ on the configuration $X$. 
The expected change, after one step of the Glauber dynamics, in the number of isomorphisms from $G$ to subgraphs of $X$ containing the edge $e$ can be counted by first negating the expected loss in number when removing a random edge $e'$ (leaving the configuration unchanged if $e'$ was not present), and then adding the expected number of graphs created by including a random edge $e'$. This gives
\begin{equation}\begin{split}\label{e:DeltaN}
  &\E\left[ \frac{N_G(X(1),e) - N_G(X(0),e)}{n^{|V|-2}}\right]\\& = \frac{1}{n^{|V|-2}}\left[-{n\choose 2}^{-1}(|E|-1)N_G(X,e)+  {n\choose 2}^{-1} \sum_{e'\neq e} N_G(X,e,e')  \P(X_{e'}(1)=1|e' \hbox{ updated})\right] \,.
\end{split} \end{equation}

Now, we may upper bound the probability of including an edge using Lemma~\ref{l:transitionProbabilities} and the definition of $\r$:
\begin{equation} \begin{split} \label{e:ProbIncludeEdge}
\P(X_{e'}(1)=1|e' \hbox{ updated}) & = \frac{\exp(\partial_e H(X))}{\exp(\partial_e H(X))+1}  \\ &= \frac{\exp\left( \sum_i \beta_i \frac{N_{G_i}(X,e)}{n^{|V|-2}}\right)}{\exp\left( \sum_i \beta_i \frac{N_{G_i}(X,e)}{n^{|V|-2}}\right)+1}
  \\ & 
  \leq (1+o(1)) \frac{\exp\left( \sum_i \beta_i \frac{N_G(K_n,e)\rei}{n^{|V|-2}}\right)}{\exp\left( \sum_i \beta_i \frac{N_G(K_n,e)\rei}{n^{|V|-2}}\right)+1}
  \\ & =(1+o(1))   \phi(\r)\,.
\end{split} \end{equation}
Next, by Lemmas~\ref{lemma:sbg-count}  and~\ref{l:countingSubgraphsSum} and the definition of $\r$, we have
\begin{equation}\begin{split}\label{e:edgeSumEstimate}
  \sum_{e'}\NGee 
  & =\sum_{\al} N_{G_\al}(X,e)\\
  & \leq \sum_{\al}  N_{G_\al}(K_n,e)(\r)^{|E|-2}\\
  & = \sum_{e'\neq e} N_G(K_n,e,e')(\r)^{|E|-2}
  \\ & = |E|(|E|-1) 2 n^{|V|-2} (\r)^{|E|-2}(1+o(1))\,.
\end{split} \end{equation}
Using the estimates \eqref{e:ProbIncludeEdge} and \eqref{e:edgeSumEstimate}, equation \eqref{e:DeltaN} gives
\begin{align*}
  &\E\left[ \frac{N_G(X(1),e) - N_G(X(0),e)}{n^{|V|-2}}\right]\\ &\leq (1+o(1)) \frac{1}{n^{|V|-2} {n\choose 2}}\left[ -(|E|-1)N_G(X,e) + \phi(\r) 2 |E| (|E|-1) n^{|V|-2}(\r)^{|E|-2} \right] \\
  &=   (1+o(1)) \frac{1}{n^{|V|-2}{n\choose 2}}\left[ -(|E|-1)2|E| n^{|V|-2} r(G,e)^{|E|-1} + \phi(\r) 2 |E| (|E|-1) n^{|V|-2}(\r)^{|E|-2} \right]
  \\ & =(1+o(1)) 
  \frac{2}{{n\choose 2}} |E| (|E|-1)  [-r(G,e)^{|E|-1}+\phi(\r)(\r)^{|E|-2}]\,.
\end{align*}

\end{proof}

\begin{lemma}\label{l:randomWalk}
  Let $p^*$ be a solution of the equation $\phi(p)=p$ with $\phi'(p^*)<1$, and let $\bar{p}$ be the least solution greater than $p^*$ of the equation $\phi(p)=p$ if such a solution exists or 1 otherwise. Let the initial configuration be $X(0)$, with $p^*+\mu\leq \r(X(0))\leq \bar{p}-\mu$ for some $\mu>0$. Then there is a $\delta,c>0$, depending only on $\mu, L$ and $\phi$, so that after $T=c n^2$ steps of the Glauber dynamics, it holds that $\r(X(T))\leq \r(X(0))-\delta$ with probability $1-e^{-\Omega(  n)}$. 
\end{lemma}

\begin{proof}
  The lemma is proved by coupling each of the random variables $N_G(X(t),e)$ (one for each edge $e$ and graph $G$), with an independent biased random walk. 
  
  Choose $\eps, \delta>0$ so that for any $r\in [p^*+\mu,\bar{p}-\mu-\delta]$, \begin{equation}\label{e:r_range}(r-2\delta)^{|E|-1} > \phi(r+\delta)(r+\delta)^{|E|-2}+\eps\,. \end{equation}
  It follows by Lemma~\ref{l:driftCalculation} that if $r_G(X(t),e)\geq \r(X(0))-2\delta$ and $\r(X(t))\leq \r(X(0))+\delta$, then for sufficiently large $n$, $$\E\left[ \frac{N_G(X(t+1),e) - N_G(X(t),e)}{n^{|V|-2}}\right]\leq -\gamma / n^2$$ for some $\gamma>0$ depending only on $\phi$, $\delta$, and $\eps$.  
 Using this negative drift we bound the probability that any of the random variables $r_G(X(t),e)$ exceed $\r(X(0))+\delta$ before time $T$.

  Define the event $$A_t(\delta)=\bigcap_{e,G}\{r_G(X(t),e) \leq \r(X(0))+\delta \}\,,$$ and put 
 \begin{align*}&D_{t}(e,G,\delta)
 =A_t\cap \{\r(X(0))-2\delta \leq r_G(X(t),e) \leq \r(X(0))+\delta\} \,,
 \end{align*}
 and $$B_{t_1,t_2}(e,G,\delta)= \left(\bigcap_{t_1\leq t<t_2}D_t\right)\cap \{ r_G(X(t_2),e)-r_G(X(t_1),e)> \delta/2 \}\,.$$
  $B_{t_1,t_2}(e,G,\delta)$ is the event that all the edge statistics $r_{G'}(X(t),e')$ behave well starting at time $t_1$ up to and including time $t_2-1$, and the statistic $r_G(X(t),e)$ increases by at least $\delta/2$ in the time period from $t_1$ to $t_2$. 
  
 The event that some $r_G(X(\tau),e)$ exceeds $\r(X(0)) +\delta$ at some time $\tau$, $1\leq \tau \leq T$, is contained in the event $\bigcup_{e,G}\bigcup_{0\leq t_1<t_2\leq T}  B_{t_1,t_2} (e,G,\delta)$. The next claim bounds the probability of the bad event for a particular choice of edge $e$ and graph $G$ and the proof of this lemma follows.

\begin{claim}
  The probability of the event $\bigcup_{0\leq t_1 <  t_2 \leq T} B_{t_1,t_2}(e,G,\delta)$ is bounded as
  \begin{equation}
     \P\left(\bigcup_{0\leq t_1 < t_2\leq T} B_{t_1, t_2}(e,G,\delta)\right)\leq e^{-\Omega (n)}\,.
  \end{equation}
\end{claim}
\begin{proof}

For all $X$ we have $N_{G_i}(X,e,e') \leq N_{G_i}(K_n,e,e')$.  The term $N_{G_i}(K_n,e,e') $ is the number of graphs $G_i$ in the complete graph containing both $e$ and $e'$.  In the case that the two edges $e$ and $e'$ share a vertex they define 3 vertices, which leaves at most $|V_i|-3$ remaining vertices to be chosen.  It follows that $N_{G_i}(K_n,e,e')\leq O(n^{|V_i|-3})$ and so
\begin{equation}\label{e:smallDiff}
N_{G_i}(X,e,e')  = O(n^{-1}).
\end{equation}
Note that an adjacent edge $e'$ is only chosen with probability $O(n^{-1})$.  When $e$ and $e'$ do not share an edge then
\begin{equation}\label{e:smallDiff2}
N_{G_i}(X,e,e')  = O(n^{-2}).
\end{equation}

Although the claim concerns the random variable $r(G,e)$, we will work with the related random variable $$Y_t= \frac{N_G(X(t),e)}{n^{|V|-2}}\,.$$ 
The first step is to compute a bound on the moment generating function of $$S_{t_1,t_2}=\sum_{t=t_1+1}^{t_2}(Y_t-Y_{t-1}+\frac{\gamma}{2n^2})\ind(D_{t-1}(e,G,\delta))\,.$$ The random variable $S_{t_1,t_2}$ is the change in $Y_i$ from time $t_1$ to $t_2$ while all the edge statistics are within the appropriate interval, shifted by $\frac{\gamma}{2n^2}$ per time step. Clearly we have the containment \begin{equation}\label{e:eventContainmentS}B_{t_1,t_2}(e,G,\delta)\subseteq \{S_{t_1,t_2}\geq \delta/2\}\,.\end{equation} 
We have $$\E\left[e^{\theta S_{t_1,t_2}}\right]= \E\left[e^{\theta S_{t_1,t_2-1}} \E\left(e^{\theta(Y_{t_2}-Y_{t_2-1}+\frac{\gamma}{2n^2})\ind(D_{t-1}(e,G,\delta))}|\F_{t_2-1}\right)\right]\,.$$
 From Lemma~\ref{l:driftCalculation} and equation \eqref{e:r_range} it follows that $\E(Y_t-Y_{t-1}\ind(D_{t-1}(e,G,\delta))|\F_{t-1})\leq -\gamma/n^2$. Recalling that with probability $1-O(n^{-1})$ it holds that $|Y_t-Y_{t-1}|=O(n^{-2})$, and it always holds that $|Y_t-Y_{t-1}|=O(n^{-1})$, we have
\begin{align*}
 &E\left(e^{\theta(Y_{t_2}-Y_{t_2-1}+\frac{\gamma}{2n^2})\ind(D_{t-1}(e,G,\delta))}|\F_{t_2-1}\right) \\
 & =
 \sum_{k=0}^\infty \E \left[ \frac{\theta (Y_{t_2}-Y_{t_2-1}+\frac{\gamma}{2n^2})^k}{k!}\ind(D_{t-1}(e,G,\delta))^k \big| \F_{t-1}\right] \\ 
 &\leq
 1-\ind(D_{t-1}(e,G,\delta))\frac{\gamma \theta}{2n^2}\\
 & \quad +\ind(D_{t-1}(e,G,\delta))\theta^2\E\left[(Y_t-Y_{t-1})^2\sum_{k=2}^\infty \frac{(\theta(Y_t-Y_{t-1}+\frac{\gamma}{2n^2}))^{k-2}}{k!}  \big| \F_{t-1}\right]
 \\&= 1-\ind(D_{t-1}(e,G,\delta))\left(\frac{\gamma \theta}{2n^2}+ O\left(\frac{\theta^2}{n^3}\right)\right)\,.
\end{align*}
Thus, when we take $\theta=cn$ for sufficiently small $c$ we have that
\[
\E\left(e^{\theta(Y_{t_2}-Y_{t_2-1}+\frac{\gamma}{2n^2})\ind(D_{t-1}(e,G,\delta))}|\F_{t-1}\right) \leq 1
\]
and so 
\begin{align*}
  \E\left[e^{\theta S_{t_1,t_2}}\right] \leq \E\left[e^{\theta S_{t_1,t_2-1}}\right]\leq 1  \,,
\end{align*}
where the second inequality follows by iterating the argument leading to the first inequality.  We can choose $\alpha>0$ depending only on $L$ and $\delta$ such that for any graph in $\{G_\lambda\}$,
\[
\alpha <   \sup_{x\in [p^*,1]}\{ (x+\delta/2)^{|E|-1} - (x)^{|E|-1}  \}.
\]
This gives the estimate
\begin{align}
  \P(S_{t_1,t_2} \geq \alpha) &\leq e^{-c\alpha n}\E\left[e^{\theta(Y_t-Y_0)}\right]=  e^{-\Omega(n)}\,.
\end{align} 
and so
\begin{equation}\label{e:S_estimate}
\P( r_G(X(t_2),e)-r_G(X(t_1),e)> \delta/2 ) = e^{-\Omega(n)}.
\end{equation}

We may now apply \eqref{e:S_estimate} to equation \eqref{e:eventContainmentS}, resulting in
\begin{equation}
    \P\left(\bigcup_{0\leq t_1 \leq t \leq t_2\leq T} B_{t_1, t_2}(e,G,\delta)\right)\leq T^2 e^{-\Omega(n)}(1+o(n))=e^{-\Omega( n)}\,,
  \end{equation} which proves the claim.
\end{proof}

Next, we argue that if all of the random variables $r_G(X(t),e)$ remain below $\r+\delta$, then each random variable actually ends below $\r-\delta$ with exponentially high probability. 
We prove this by showing that each random walk actually reaches $\r-2\delta$, and then by the claim has exponentially small probability of increasing to $\r-\delta$. 
Suppose that for some $e,G$, $r_G(X(0),e)\geq \r-2\delta$.  Then for $T=cn^2$, 
\begin{align*}&\P(r_G(X(t),e)\geq \r-2\delta \text{ for } 1\leq t\leq T) \\ &\leq \P(r_G(X(t),e)\geq \r-2\delta \text{ for } 1\leq t\leq T, \cap_{1\leq t\leq T} A_t(\delta))+e^{-\Omega(n)}
\\ &\leq \P(r_G(X(t),e)\geq \r-2\delta \text{ for } 1\leq t\leq T, \cap_{1\leq t\leq T} D_t(e,G,\delta))+e^{-\Omega(n)}
\\ &\leq \P(S_{1,T}\geq -1 + \frac{\gamma c}{2})+e^{-\Omega(n)}\,,\end{align*} where the last step follows since each of the $T$ increments in $S_{1,T}$ contribute $\gamma/2n^2$ on the event $\cap_{1\leq t\leq T} D_t(e,G,\delta)$. Choosing $c\geq 3/\gamma$ and using the estimate on the deviation of $S_{t_1,t_2}$ \eqref{e:S_estimate} gives
$$\P(r_G(X(t),e)\geq \r-2\delta \text{ for } 1\leq t\leq T)\leq e^{-\Omega(n)}\,.$$
Finally, we have 
\begin{align*}&\P(r_G(X(T),e)\geq \r-\delta)\\ &\leq \P(r_G(X(T),e) \geq \r-\delta, r_G(X(t),e)< \r-2\delta \text{ for some } t\in [1,T])+e^{-\Omega(n)}
\\ & \leq \P(\cup_{1\leq t_1\leq T} B_{t_1,T}(e,G,\delta))+e^{-\Omega(n)} \leq e^{-\Omega(n)}\,.
\end{align*}
The union bound on probabilities applied over the set of edges $e$ and graphs $G$ completes the proof of Lemma~\ref{l:randomWalk}.
\end{proof}

The following lemmas follow immediately from iterating Lemma \ref{l:randomWalk}.

\begin{lemma}\label{l:highTempBurnIn}
In the high temperature phase for any $\eps>0$ there is $c>0$ such that for any initial configuration $X(0) = x$,  when $t\geq cn^2$ we have
\begin{align*}
\P(\r(X(t)) \geq p^*+\eps|X(0) = x) \leq e^{-\Omega(n)},\\
\P(\rmin(X(t)) \leq p^*-\eps|X(0) = x) \leq e^{-\Omega(n)}.
\end{align*}
\end{lemma}

\begin{lemma}\label{l:lowTempBurnIn}
In the low temperature phase suppose that $p^*$ is a solution to $p^=\phi(p^)$ and $\phi'(p^*)<1$.  There exists an $\epsilon>0$ such that if for some initial configuration $X(0)$ we have that $\r(X(0))\leq p^*+\eps$ and $\rmin(X(0))\geq p^*-\eps$ then for some $\alpha>0$
\begin{align*}
\P\left(\sup_{0<t<e^{\alpha n}} \r(X(t)) \geq p^* + 2\eps\right)  &\leq e^{-\Omega(n)},\\
\P\left(\inf_{0<t<e^{\alpha n}} \r(X(t)) \leq p^* - 2\eps\right)  &\leq e^{-\Omega(n)}.
\end{align*}
\end{lemma}

\subsection{Path coupling}
\begin{lemma}\label{l:pathCoupling}
Let $p^*\in [0,1]$ be a solution of the equation $\varphi(p)=p$ and suppose $0<\varphi'(p^*)<1$.  There exists $\epsilon,\delta>0$ sufficiently small and such that the following holds.  Suppose that $X^+(0)\geq X^-(0)$ are two configurations that differ at exactly one edge $e$.  Suppose further that for all graphs $G$ with at most  $L$ vertices and all edges $e'$
\begin{equation}\label{e:concentrationOfR}
\left |r(G,e') - p^* \right| < \epsilon.
\end{equation}
Then for sufficiently large $n$ a single step of the Glauber dynamics can be coupled so that
\[
E d_H(X^+(1), X^-(1)) \leq 1- \delta n^{-2}.
\]
\end{lemma}

\begin{proof}

We take the standard monotone coupling.  Suppose that an edge $e'\neq e$ is chosen to be updated by the Markov chain.  Then
\begin{equation}\label{e:probAddEdge}
P(X^\pm_{e'}(1) = 1) = \frac{\exp(\partial_{e'} H(X^\pm(0)))}{1 + \exp(\partial_{e'} H(X^\pm(0)))}.
\end{equation}
Since
\[
\partial_{e'} H(X^\pm(0)) = \sum_{i=1}^s \frac{ \beta_i N_{G_i}(X^\pm(0),e') } {n^{|V_i|-2}}
\]
by Lemma \ref{lemma:sbg-count}  and equation \eqref{e:concentrationOfR} we have that for large enough $n$,
\begin{equation}\label{e:approxPsi}
\partial_{e'} H(X^\pm(0)) \leq  \sum_{i=1}^s \frac{ \beta_i (p^* + \epsilon)^{|V_i|-2}N_{G_i}(K_n,e') } {n^{|V_i|-2}} =  \Psi(p^*+ \epsilon).
\end{equation}
Similarly 
\[
0 \leq (1-o(1))\Psi(p^*-\epsilon) \leq \partial_{e'} H(X^\pm(0))
\]
and so it follows that for any $\epsilon'>0$ that for large enough $n$ and for small enough $\epsilon$ we have that
\begin{equation}\label{e:approxDeriv}
\left .  \frac{d} {dx}\frac{e^x}{1+e^x} \right |_{\partial_{e'} H(X^+(0)) } \leq    (1+\epsilon')   \left .\frac{d} {dx}\frac{e^x}{1+e^x} \right |_{\Psi(p^*) }.
\end{equation}

We now bound the sum of the  $\partial_{e}  \partial_{e'} H(X^+(0))$ terms
\begin{align*}
\sum_{e' \neq e} \partial_{e}  \partial_{e'} H(X^+(0)) &=  \sum_{e' \neq e}\sum_{i=1}^s \frac{ \beta_i N_{G_i}(X^+(0),e,e') } {n^{|V_i|-2}} \\
&=  \sum_{i=1}^s \frac{ \beta_i \sum_{\al\in E(G_i)\atop \al\neq e} N_{(G_i)_\al}(X^+(0),e)  } {n^{|V_i|-2}} \\
&\leq  \sum_{i=1}^s \frac{ \beta_i \sum_{\al\in E(G_i)\atop \al\neq e} (p^* + \epsilon)^{|E_i|-2}N_{(G_i)_\alpha}(K_n,e) } {n^{|V_i|-2}} \\
&= \sum_{i=1}^s   \sum_{e' \neq e} \frac{ \beta_i (p^* + \epsilon)^{|E_i|-2} N_{G_i}(K_n,e,e') } {n^{|V_i|-2}}\,,
\end{align*}
where the second and fourth lines follow from Lemma \ref{l:countingSubgraphsSum} and the inequality follows from Equation \eqref{e:concentrationOfR}.  By Lemma \ref{lemma:sbg-count} we have that
\begin{equation}\label{e:DoubleDerivBound}
\sum_{e' \neq e} \partial_{e}  \partial_{e'} H(X^+(0)) \leq (1+o(1)) \sum_{i=1}^s 2  |E_i| (|E_i|-1) (p^* + \epsilon)^{|E_i|-2} = (1+o(1))\Psi'(p^* + \epsilon)\,.
\end{equation}


By Taylor series for small $h$ we have that $\frac { e^{x+h} } {1+e^{x+h}} - \frac { e^x } {1+e^x} \leq \left .  \frac{d} {dx}\frac{e^x}{1+e^x} \right |_{x} (h + O(h^2))$ and so using Equation \eqref{e:probAddEdge},
\begin{align}\label{e:probDiff}
P(X^+_{e'}(1) = 1) - P(X^-_{e'}(1) = 1) & = \left. \frac{d} {dx}\frac{e^x}{1+e^x} \right |_{\partial_{e'} H(X^+(0)) }  
\cdot   \left( \partial_{e}  \partial_{e'} H(X^+(0))  +  O((\partial_{e}  \partial_{e'} H(X^+(0))^2) \right)\\
&\leq (1+\epsilon') (1+o(1))\partial_{e}  \partial_{e'} H(X^+(0))    \left .\frac{d} {dx}\frac{e^x}{1+e^x} \right |_{\Psi(p) } \,,
\end{align}
using equations \eqref{e:approxDeriv} and the fact that by equation  \eqref{e:smallDiff} we have that $\partial_{e}  \partial_{e'} H(X^+(0)) =O(n^{-1})$.

Each edge $e'$ has probability ${n\choose 2}^{-1}$ of being updated and if edge $e$ is chosen to be updated then the number of disagreements is 0.  It follows by equations \eqref{e:probDiff} and \eqref{e:DoubleDerivBound} that for any $\epsilon''>0$
\begin{align*}
E d_H(X^+(1), X^-(1)) &\leq 1 - {n\choose 2}^{-1}\left[ 1 -  \sum_{e'\neq e}  (1+\epsilon') (1+o(1))\partial_{e}  \partial_{e'} H(X^+(0))    \left .\frac{d} {dx}\frac{e^x}{1+e^x} \right |_{\Psi(p) }   \right]\\
& \leq 1 - {n\choose 2}^{-1}\left[ 1 -  (1+\epsilon') (1+o(1))\Psi'(p^* + \epsilon)  \left .\frac{d} {dx}\frac{e^x}{1+e^x} \right |_{\Psi(p) }   \right]\\
&\leq  1 - {n\choose 2}^{-1}\left[ 1 - (1+\epsilon'') (1+o(1))\varphi'(p^*)   \right]
\end{align*}
provided that $\epsilon,\epsilon'$ are sufficiently small.  The result follows, since $\varphi'(p^*)<1$.

\end{proof}

{\bf Proof of Theorem \ref{theo:high}}

We begin by proving the high temperature phase using a coupling argument.  Let $X^+(t)$ and $X^-(0)$ be two copies of the Markov chain started from the complete and empty configurations, respectively, and coupled using the monotone coupling. Since this is a monotone system, it follows that if $P(X^+(t)\neq X^-(t))<e^{-1}$, then $t$ is an upper bound on the mixing time.  The function $\varphi$ satisfies the hypothesis of Lemma \ref{l:pathCoupling}, so choose $\epsilon$ and $\delta$ according to the lemma.  Let property $\mathcal{A}_t$ be the event that for all graphs $G$ with at most  $L$ vertices and all edges $e$
\begin{equation}\label{e:concentrationOfR2}
\left |r(G,e) - p^* \right| < \epsilon
\end{equation}
for both $X^+(t)$ and $X^-(0)$.  By Lemma \ref{l:highTempBurnIn} we have that if $t\geq cn^2$, then $P(\mathcal{A}_t) \geq1-e^{-\alpha n}$.

Since the subgraph counts $N_G(X,e)$ are monotone in $X$, if $X^+(t)$ and $X^-(t)$ both satisfy equation \eqref{e:concentrationOfR2}, then there exists a sequence of configurations $X^-(t) = X^0 \leq X^1 \leq \ldots \leq X^{d}=X^+(t)$, where $d=d_H(X^+(t),X^-(t))$, each pair $X^i,X^{i+1}$ differ at exactly one edge, and each $X^i$ satisfies equation \eqref{e:concentrationOfR2}.  Such a sequence is constructed by adding one edge at a time to $X^-(t)$ until $X^+(t)$ is reached.   Applying path coupling to this sequence, we have that by Lemma \ref{l:pathCoupling}
\[
E \left[ d_H(X^+(t+1),X^-(t+1)) | X^+(t),X^-(t+1) ,\mathcal{A}_t \right] \leq (1- \delta n^{-2}) d_H(X^+(t),X^-(t)).
\]
Since $d_H(X^+(t),X^-(t))\leq {n\choose 2}$, we have the inequality
\begin{align*}
E \left[ d_H(X^+(t+1),X^-(t+1)) \right]  &\leq (1- \delta n^{-2}) E\left[ d_H(X^+(t),X^-(t)) | \mathcal{A}_t  \right]P(\mathcal{A}_t) + {n\choose 2}(1-P(\mathcal{A}_t))\\
&\leq (1- \delta n^{-2}) E\left[ d_H(X^+(t),X^-(t)) \right]+ {n\choose 2}e^{-\alpha n}\,.
\end{align*}
Iterating this equation, we get that for $t>C'n^2$,
\begin{align*}
E\left[ d_H(X^+(t),X^-(t)) \right] &\leq (1- \delta n^{-2})^{t-Cn^2}{n\choose 2} + e^{-\alpha n}{n\choose 2}\sum_{j=C'n^2}^t (1- \delta n^{-2})^{t-j}\\
&\leq \exp(- \delta n^{-2}(t-Cn^2))n^2 + e^{-\alpha n} \frac1{\delta}{n\choose 2}n^2.
\end{align*}
Then for any $\epsilon'>0$, when $t> \frac{2 + \epsilon'}{\delta} n^2\log n$ we have that for large enough $n$,
\[
E\left[ d_H(X^+(t),X^-(t)) \right] = o(1)
\,. \]
It follows by Markov's inequality that $P(X^+(t) \neq X^-(t)) =o(1)$, which establishes that the mixing time is bounded by $ \frac{2 + \epsilon'}{\delta} n^2\log n$.


\subsection{Slow mixing for local Markov chains in low-temperature regime}

We will use the following conductance result, which is taken from \cite{dyer-frieze-jerrum02} (Claim 2.3):
\begin{claim}\label{claim:ConductanceBound}
  Let $\mathcal{M}$ be a Markov chain with state space $\Omega$, transition matrix $P$, and stationary distribution $\pi$. Let $A\subset \Omega$ be a set of states such that $\pi(A)\leq \frac{1}{2}$, and $B\subset \Omega$ be a set of states that form a ``barrier" in the sense that $P_{ij}=0$    whenever $i\in A\setminus B$ and $j\in A^c\setminus B$. Then the mixing time of $\mathcal{M}$ is at least $\pi(A)/8\pi(B)$.
\end{claim}

Using this result we prove slow mixing for any local Markov chain.


{\bf Proof of Theorem \ref{theo:low}}

\begin{proof}
Suppose $p_1$ and $p_2$ are solutions of the equation $\varphi(p)=p$ with $\varphi'(p_1)<1,\varphi'(p_2)<1$, and choose $\eps>0$ sufficiently small so that $\varphi(p)<p$ for $p\in (p_i,p_i+3\eps]$ and $\varphi(p)>p$ for $p\in [p_i-3\eps,p_i)$, for $i=1,2$. Let $$A_i=\{X:\r(X)\leq p_i+\eps \text{ and } \rmin(X)\geq p_i-\eps\},\quad i=1,2\,,$$ and suppose the set $A_1$ has smaller probability (switching the labels $p_1$ and $p_2$ if necessary), so $\pi(A_1)\leq \frac{1}{2}$. We note that for large enough $n$,  $\pi(A_i)>0$ since with high probability an Erd\"os-Reny\'i random graph $G(n,p_i)$ is in $A_i$.
In the remainder of the proof we will omit the subscript, i.e. let $A=A_1$ and $p=p_i$. Now, clearly the set
$$B=\{X:p+\eps<\r(X)\leq p+2\eps \text{ or } p-\eps>\rmin(X)\geq p-2\eps\}$$ forms a barrier (for sufficiently large $n$) between the sets $A$ and $A^c$ for any Markov chain that updates only $o(n)$ edges per time-step, since each edge update can change each of $\r$ and $\rmin$ by at most $O\left(\frac{1}{n}\right)$. 

It remains only to bound the relative probabilities of the sets $A$ and $B$. Let $C=A^c\setminus B$, and let $t=cn^2$ such that Lemma~\ref{l:randomWalk} holds. Then 
\begin{equation}\label{e:transitionBound1}
  \P(X(t)\in C| X(0)\in B)=e^{-\Omega(n)}
\end{equation}
and
\begin{equation}\label{e:transitionBound2}
  \P(X(t)\in B |X(0) \in A\cup B)=e^{-\Omega(n)}\,.
\end{equation}
Let the configuration $X(0)$ be drawn according to the Gibbs measure $\pi=p_n$ defined in Equation \eqref{eqn:gibbs}, and let $X(t)$ be the configuration resulting after $t$ steps of the Glauber dynamics. Because the Glauber dynamics has stationary distribution $\pi$, $X(t)$ has the same distribution as $X(0)$. By the reversibility of the Glauber dynamics and the estimate \eqref{e:transitionBound1} we have
\begin{equation}\begin{split}\label{e:slowMixing1}
  \P(X(t)\in B,X(0)\in C)&= \P(X(t)\in C, X(0)\in B)\\ & = \P(X(t)\in C| X(0)\in B)P(X(0)\in B) \\ &=e^{-\Omega(n)} \P(X(0)\in B)\,.
\end{split}\end{equation}
Similarly, using \eqref{e:transitionBound2},
\begin{equation}\begin{split}\label{e:slowMixing2}
  \P(X(t)\in B, X(0)\in A\cup B) &= \P(X(t)\in B |X(0) \in A\cup B) \P(X(0) \in A\cup B) \\ & \leq e^{-\Omega(n)} \P(X(0) \in A\cup B) \\ & =  e^{-\Omega(n)} (\P(X(0)\in A)+\P(X(0)\in B))\,.
\end{split}\end{equation}
Combining \eqref{e:slowMixing1} and \eqref{e:slowMixing2}, we have
\begin{equation}
  \begin{split}
    \pi(B) &= \P(X(t)\in B) \\ &= \P(X(t)\in B,X(0)\in C) + \P(X(t)\in B, X(0)\in A\cup B)
   \\& \leq e^{-\Omega(n)} (\P(X(0)\in A)+2\P(X(0)\in B))
     \\&=  e^{-\Omega(n)} (\pi(A)+2\pi(B))\,,
  \end{split}
\end{equation}
which, upon rearranging, gives
\begin{equation}
  \pi(B)\leq \frac{e^{-\Omega(n)}}{1-2e^{-\Omega(n)}}\pi(A)\,.
\end{equation}
Together with Claim~\ref{claim:ConductanceBound} this completes the proof.
\end{proof}

\section{Asymptotic independence of edges and weak pseudo-randomness}

Our burn in proof in the high temperature regime shows that with high probability all the $\rGe$ are close to $p^*$, the fixed point of $\phi(p)=p$.  A consequence is that for any collection of edges $e_1,\ldots,e_j$ the events $x_{e_i}$ are asymptotically independent and distributed as $\hbox{Bernoulli}(p^*)$. A consequence of the asymptotic independence of the edges is that with high probability a graph sample from the exponential random graph distribution is \emph{weakly pseudo-random}, as defined in \cite{KrSu}. As such, the exponential random graph model is extremely similar to the basic Erd\H{o}s-R\'enyi  random graph.  Since exponential random graphs were introduced to model the phenomenon of increased numbers of small subgraphs like triangles, this result proves that the model fails in it's stated goal.

\begin{Theorem} Let $X$ be drawn from the exponential random graph distribution in the high temperature phase. 
  Let $e_1,\dots,e_k$ be an arbitrary collection of edges with associated indicator random variables $x_{e_i}=\ind(e_i\in X)$. Then for any $\eps>0$, there is an $n$ such that for all $(a_1,\dots,a_k)\in \{0,1\}^k$ the random variables $x_{e_1},\dots,x_{e_k}$ satisfy 
  $$\left|\P(x_1=a_1,\dots,x_k=a_k)-(p^*)^{\sum a_i}(1-p^*)^{k-\sum a_i}\right|\leq \frac{\eps}{n^{|V|}}\,.$$
  Thus, the random variables  $x_{e_1},\dots,x_{e_k}$ are asymptotically independent.
\end{Theorem}
\begin{proof}
Fix $\eps>0$.  Let $S\subseteq [k]$ and let $x_S=\{x_{e_i}:i\in S\}$ and $x_{S^c}=\{x_{e_i}:i\in [k]\setminus S\}$. Then by the inclusion-exclusion principle, we have
\begin{equation}\label{e:inclusionExclusion}
  \P(x_S=1,x_{S^c}=0)=\sum_{T\subseteq [k]-S} (-1)^{|T|} \P(x_{S\cup T}=1)\,.
\end{equation}
We argue next that each probability in the preceding sum satisfies $\P(x_{S\cup T}=1)\approx (p^*)^{|S\cup T|}$.
Let $$A=\{X:\r(X)\leq p^*+\eps' \text{ and } \rmin(X)\geq p^*-\eps'\}\,,$$
where $\eps'$ is to be specified later.
Consider the subgraph $G_T$ formed by the edges in the set $T$. For all configurations $X\in A$, $|r_{G_T}(X,e)-p^*|\leq \eps'$, which gives \begin{equation}\label{e:graphCountsTypicalER}\left|N_{G_T}(X,e)-(p^*)^{|E|-1}2|E|n^{|V|-2}\right| \leq \eps''\end{equation} 
for sufficiently large $n$. By considering the graph consisting of two disjoint edges, we have that the number of edges in a configuration $X\in A$ satisfies \begin{equation}\label{e:edges}\left|N_{\text{edge}}(X)-p^*{n\choose 2}\right|\leq \eps''\,.\end{equation} 
Note that $$\sum_{e\in X} N_G(X,e)=|E|N_G(X)\,,$$ and summing Equation \eqref{e:graphCountsTypicalER} over the edges in $X$ and using \eqref{e:edges}, this gives 
\begin{equation}\label{e:SubgraphCountsPseudo}\left|N_{G_T}(X)- (p^*)^k n^{|V|}\right|\leq \frac{\eps}{2^k}\end{equation} for a sufficiently small choice of $\eps'$ in the definition of the set $A$. By symmetry, each of the subgraphs $G_T$ is equally likely to be included in the configuration $X$, and there are $n^{|V|}$ possible such subgraphs, so $\P(x_T=1)=\frac{N_{G_T}(X)}{n^{|V|}}$. It follows that $|\P(x_T=1|X\in A)-(p^*)^k|\leq  \frac{\eps}{2^k n^{|V|}}$. Recall that $\P(A)=1+o(1)$ in the high temperature phase. Thus, for any set of edges $T\subseteq[k]$, it holds that $$\left|\P(x_T=1)-(p^*)^{|T|}\right|=(1+o(1))\left|\P(x_T=1|X\in A)-(p^*)^{|T|}\right|\leq(1+o(1))\frac{\eps}{2^k}\,.$$ Hence
$$\left|\sum_{T\subseteq [k]-S} (-1)^{|T|} \P(x_{S\cup T}=1)
-\sum_{T\subseteq [k]-S} (-1)^{|T|} (p^*)^{|S\cup T|}\right|\leq \frac{\eps}{n^{|V|}}$$ for sufficiently large $n$, and the desired result follows from \eqref{e:inclusionExclusion} and the fact that
\begin{align*}
  \sum_{T\subseteq [k]-S} (-1)^{|T|} (p^*)^{|S\cup T|}
  &=(p^*)^{|S|}\sum_{q=0}^{k-|S|}{k-|S|\choose q} (-p^*)^q
  \\&=(p^*)^{|S|}(1-p^*)^{k-|S|}
  \,.
\end{align*}

\end{proof}

We can also show that an exponential random graph is weakly pseudo-random with high probability.  This means that a collection of equivalent conditions are satisfied; we briefly mention only a few of them (see the survey on pseudo-random graphs \cite{KrSu}). We will use a different subgraph count than before: for a graph $G$ let $N^*_G(X)$ be the number of labeled induced copies of $G$ in $X$. This is different than the counts $N_G(X)$ in that it requires edges missing from $G$ to also be missing in the induced graph in $X$.
A graph $X$ is weakly pseudo-random if it satisfies one of the following (among others) equivalent properties:
\begin{enumerate}
  \item For a fixed $l\geq 4$ for all graphs $G$ on $l$ vertices, $$N^*_G(X)=(1+o(1))n^l (p)^{|E(G)|}(1-p)^{{l\choose 2}-|E(G)|}\,.$$
  \item $N_{\text{edges}}(X)\geq \frac{n^2 p}{2} + o(n^2)$ and $\lambda_1=(1+o(1))np$, $\lambda_2=o(n)$, where the eigenvalues of the adjacency matrix of $X$ are ordered so that $|\lambda_1|\geq |\lambda_2|\geq \cdots \geq |\lambda_n|$. 
  \item For each subset of vertices $U\subset V(X)$ the number of edges in the subgraph of $X$ induced by $U$ satisfies $E(H_X(U))=\frac{p}{2}|U|^2+o(n^2)$.
  \item Let $C_l$ denote the cycle of length $l$, with $l\geq 4$ even. The number of edges in $X$ satisfies $N_{\text{edges}}(X)= \frac{n^2 p}{2} + o(n^2)$ and the number of cycles $C_l$ satisfies $N_{C_l}(X)\leq (np)^l+o(n^l)$.
\end{enumerate}

By \eqref{e:SubgraphCountsPseudo}, for any configuration in the good set, $X\in \mathbf{G}$, the fourth condition is satisfied. This gives the following corollary. 
\begin{Corollary}[Weak pseudo-randomness]
With probability $1+o(1)$ an exponential random graph is weakly pseudo-random. 
\end{Corollary}

\bibliographystyle{plain}
\bibliography{BIBFILE}

\begin{thebibliography}{10}

\bibitem{barabasi-survey}
R.~Albert and A.~L. Barabasi.
\newblock Statistical mechanics of complex networks.
\newblock {\em Reviews of Modern Physics}, 74(1):47--97, 2002.

\bibitem{aldous-fill}
D.~Aldous and J.~Fill.
\newblock {\em Reversible {M}arkov Chains and Random Walks on Graphs}.
\newblock Book in preparation.

\bibitem{wasserman-survey}
C.~J. Anderson, S.~Wasserman, and B.~Crouch.
\newblock A p* primer: logit models for social networks.
\newblock {\em Social Networks}, 21(1):37--66, January 1999.

\bibitem{BBS08}
S.~Bhamidi, G.~Bresler, and A.~Sly.
\newblock Mixing time of exponential random graphs.
\newblock In {\em Foundations of Computer Science (FOCS)}, 2008.

\bibitem{bubley-dyer97}
R.~Bubley and M.~Dyer.
\newblock Path coupling: A technique for proving rapid mixing in {M}arkov
  chains.
\newblock In {\em {Foundations of Computer Science (FOCS)}}, pages 223--231,
  1997.

\bibitem{sourav-rnd}
S.~Chatterjee.
\newblock Exact large deviations for triangles in a random graph, 2008.
\newblock Work in progress.

\bibitem{CGW89}
F.~R.~K. Chung, R.~Graham, and R.~M. Wilson.
\newblock Quasi-random graphs.
\newblock {\em Combinatorica}, 9:345--362, 1989.

\bibitem{durrett-book}
R.~Durrett.
\newblock {\em Random Graph Dynamics}.
\newblock {Cambridge University Press}, 2006.

\bibitem{dyer-frieze}
M.~Dyer and A.~Frieze.
\newblock Randomly coloring graphs with lower bounds on girth and maximum
  degree.
\newblock {\em Random Structures and Algorithms}, 23(2):167--179, 2003.

\bibitem{dyer-frieze-jerrum02}
M.~Dyer, A.~Frieze, and M.~Jerrum.
\newblock On counting independent sets in sparse graphs.
\newblock {\em SIAM Journal Computing}, 31(5):1527--1541, 2002.

\bibitem{frank-strauss}
O.~Frank and D.~Strauss.
\newblock Markov graphs.
\newblock {\em Journal of the American Statistical Association},
  81(395):832--842, 1986.

\bibitem{newmancluster}
{J. Park and M. E. J. Newman}.
\newblock {Solution for the properties of a clustered network}.
\newblock {\em Physical Review E}, 72(2):026136.1--026136.5, August 2005.

\bibitem{KrSu}
M.~Krivelevich and B.~Sudakov.
\newblock Pseudo-random graphs.
\newblock In {\em Conference on finite and infinite sets}, pages 1--64,
  Budapest, 2005.

\bibitem{newman-survey}
M.~E.~J. Newman.
\newblock The structure and function of complex networks.
\newblock {\em SIAM Review}, 45(2):167--256, 2003.

\bibitem{newman2star}
J.~Park and M.~E.~J. Newman.
\newblock Solution of the 2-star model of a network.
\newblock {\em Physical Review E}, 70, 2004.

\bibitem{snijders2004}
T.~A. Snijders, P.~Pattison, G.~Robbins, and M.~Handcock.
\newblock New specifications for exponential random graph models.
\newblock {\em Sociological Methodology}, 36(1):99--153, December 2006.

\bibitem{wasserman-pattison}
S.~Wasserman and P.~Pattison.
\newblock Logit models and logistic regressions for social networks.
\newblock {\em Psychometrika}, 61(3):401--425, 1996.

\end{thebibliography}
\end{document}